\documentclass[12pt,4apaper]{article}
\usepackage[cp1251]{inputenc}
\usepackage[english]{babel}

\usepackage{amsmath,amsfonts,amssymb,mathrsfs,amscd,comment,latexsym,bbold}
\usepackage[matrix,arrow,curve]{xy}

\usepackage[usenames]{color}
\usepackage{colortbl}
\usepackage[dvipsnames]{xcolor}

\usepackage{enumitem}

\usepackage[top = 3cm, bottom = 4cm, left = 3cm, right = 3cm]{geometry}

\usepackage{amsthm}

\theoremstyle{plain}
\newtheorem{theorem}{Theorem}
\newtheorem*{nonum-theorem}{Theorem}
\newtheorem{proposition}{Proposition}
\newtheorem{lemma}{Lemma}[section]
\newtheorem{lemma-remark}{Lemma-remark}[section]

\newtheorem{corollary}{Corollary}
\newtheorem{nonum-corollary}{Corollary}

\newtheoremstyle{handleNumber}{}{}{\itshape}{}{}{}{\newline}{{\bf #1} \thmnote{#3}}
\theoremstyle{handleNumber}
\newtheorem*{handnum-theorem}{Theorem}

\theoremstyle{definition}
\newtheorem{definition}{Definition}
\newtheorem{denotation}{Denotation}
\newtheorem*{nonum-definition}{Definition}
\newtheorem{construction}{Construction}

\theoremstyle{remark}
\newtheorem{remark}{Remark}

\newcommand{\aff}{{\mathbb{A}}}
\newcommand{\affl}{\mathbb{A}^1}

\newcommand{\prl}{\mathbb{P}^1}
\newcommand{\pro}{{\mathbb{P}}}

\newcommand{\bO}{\mathscr{O}}
\newcommand{\bF}{\mathscr{F}}
\newcommand{\bL}{\mathscr{L}}

\newcommand{\bC}{\mathcal{C}}

\newcommand{\bCp}{{\mathcal{C}^\prime}}
\newcommand{\ovC}{{\overline C}}
\newcommand{\ovbC}{{\overline{\mathcal{C}}}}
\newcommand{\bCinf}{{\mathcal{C}_{inf}}}

\newcommand{\sbC}{\mathcal{C}^\prime}
\newcommand{\sbCinf}{\mathcal C_{inf}^\prime}
\newcommand{\ssbC}{\mathcal{C}^{\prime\prime}}
\newcommand{\ssbCinf}{\mathcal C_{inf}^{\prime\prime}}

\newcommand{\Cinf}{C_{inf}}

\begin{document}

\title{Rigidity theorem for presheaves with Witt-transfers}
\author{Andrei Druzhinin
\thanks{Research is supported by "Native towns", a social investment program of PJSC "Gazprom Neft" }
}
\date{April 12, 2017}
\newcommand{\Address}{{
  \bigskip
  \footnotesize
  \textit{address:}
  \textsc{Chebyshev Laboratory, St. Petersburg State University, 14th Line V.O., 29B, Saint Petersburg 199178 Russia}
  \par\nopagebreak
  \textit{e-mail address:}
   \texttt{andrei.druzh@gmail.com}
}}

\maketitle

\abstract{
The rigidity theorem for homotopy invariant presheaves with Witt-transfers on the category of smooth affine varieties over a field 
with characteristic not equal to $2$ is proved. Namely the isomorphism $\bF(U)\simeq \bF(x)$ where $U$ is henseliation of a variety at smooth closed point with separable residue field is proved. The rigidity for presheaves $W^i(-\times X)$ where $X$ is smooth variety and $W^i(-)$ are derived Witt-groups follows as corollary.}

\section{Introduction}

Let’s give brief review of rigidity theorems. 

In the Suslin's remarkable article \cite{SusAlgClFiledKth} (in that he proved Quillen's conjecture for K-thory) 
he proved that $K(F_0,\mathbb Z/n)\simeq K(F,\mathbb Z/n)$ for extensions of algebraically closed fields $F/F_0$ and $n$ prime to its characteristic. 
Then Gabber in \cite{Gab_RigKthhenspairs} and Gillet and Thomason in \cite{GT_HensLocRingsKth} extended this result to henselian pairs and strictly henselisation at smooth point. Namely the last one is isomorphism 
$K(k,\mathbb Z/n)\simeq K(R,\mathbb Z/n)$ where 
  $R$ is Henselisation at smooth point of variety of finite type over separably closed field $k$,
and in \cite{Sus_LocFieldsKth} Suslin proved this for $R$ being any henselian ring with valuation of height one and $k$ being its residue field. 

One can see that this proves are based on idea of using of transfers for K-theory.
In \cite{SV_SingHomofSchemes} in that Suslin and Voevodsky give definition of singular cohomolgy of algebraic variety over an arbitrary filed $k$ and prove that for variety over $\mathbb{C}$ this groups coincidences with topological singular cohomology of $X(\mathbb{C})$, they proved rigidity theorem 
for henselian local ring at smooth rational point of an algebraic variety
for any homotopy invariant presheave with transfers defined by Cor-correspondences, where $Cor(X,Y)$ is defined as free abelian group of integral cycles in $X\times Y$ finite surjective over $X$. 

Further in \cite{PY_RigidityOrientFunctor} 
(Panin and Yagunov)
in \cite{YH_RidgsomeRepesTh} 
(Yagunov Hornbostel)
and in \cite{Y_NonOrientableTheor} (Yagunov)
proved rigidity theorems for functors with another (and weaker) transfers and get regidiry for orientable cohomology theories (in particular etale cohomology, motivic cohomology, algebraic cobardism), rigidity for some class of theories representable in $\mathbb{A}^1$-homotopy category and ono-orientable cohomology theories. 
In \cite{RO_RidgSHfieldex} proved rigidity at the level of motivic stable homotopy categories.
In \cite{N_RigdOmegatr} proved rigidity for presheaves with $\Omega$-transfers. 
In \cite{B_MotRealEtHth} Bachmann proved rigidity for theories representable in stable homotopy category with inverting $\rho$ that is element in $Hom_{SH(R)}(\mathbb{1}, \mathbb{G}_m)$ additive inverse to element defined by $-1\in R$

We prove rigidity theorem for homotopy invariant presheaves with so-called Witt-transfers involved in \cite{AD_WittCor}. In distinct to the most of the above results (except last one) we don't assume that presheaf is torsion. As a corollary rigidity for presheaves $W^i(X\times -)$ where $W^i(-)$ are derived Witt-groups follows form this because these presheaves are presheaves with Witt-transfers (proposition \ref{prop_WittgrWitttr}) and they are homotopy invariant as shown in \cite{Bal_HomInvWitt}.
\begin{nonum-theorem}
Let $X$ is algebraic variety over a field $k$ with $char\,k\neq 2$ and $x$ is smooth closed point with separable residue field $k(x)/k$.
Then for any homotopy invariant presheave $\bF$ with Witt-transfers
$\bF(Spec\,\mathcal{O}^h_{X,x})\simeq \bF(x)$
\end{nonum-theorem}
As a corollary rigidity for presheaves $W^i(Y\times -)$ 
 follows form this because these presheaves are presheaves with Witt-transfers and they are homotopy invariant as shown in \cite{Bal_HomInvWitt}.
\begin{nonum-corollary}
For a smooth variety $X$ over field $k$ , $char k\neq 2$ and for a closed point $x$ on $X$ with separable residue field $k(x)/k$
and any smooth 
$Y$
$$W^i(Y\times \mathcal Spec\,O^h_{X,x})\simeq W^i(Y\times Spec\,k(x))$$
where $W^i$ $(i\in \mathbb Z/4\mathbb Z)$ derived(higher) Witt groups.
\end{nonum-corollary}
Similarly to the works noted above we conclude the rigidity theorem from the statement about sections of relative curve over $U=Spec\,\mathcal{ O}^h_{X,x}$ that coincide at closed fibre.
\begin{nonum-theorem}
Let
  $U$ be henselisation of a variety $X$ at smooth closed point with separable reside field over a base field $k$, $char\, k \neq 2$.
Let 
  $\ovbC$ be  any relative projective variety over $U$ with fibres of dimension one 
    (i.e. smooth projective morphism $\ovbC\to U$ of relative dimension 1).
Let $\bCinf\subset \ovbC$ be closed subsceme that is finite over $U$ and such that open complement $\ovbC-\bCinf$ is smooth over $U$. 
Then for any two sections 
  $S_0,S_1\colon U\to \ovbC$ 
    coinciding in closed fibre, i.e. such that $S_0(u) = S_1(u)=x\in \ovC$ 
    (where $u\in U$ denotes closed point and $C\subset \bC$ denotes closed fibre of $\bC$)
$$\xymatrix{
\bC & C\ar@{^(->}[l]  &\\
U\ar@<-1ex>[u]_{S_1}\ar@<1ex>[u]^{S_0} & u\ar@{^(->}[l]\ar[u]^x
&,}$$
homotopy classes of sections $S_0$ and $S_1$ in the group $\overline{WCor}(U,\bC)$
are the same, i.e.
$$S_0,S_1\colon U\to \bC, S_0(u)=S_1(u)\implies [S_0] = [S_1] \in \overline{WCor}(U,\bC).$$
\end{nonum-theorem}

The text is organized into three sections. 
In the first one we recall definition  of Witt-correspondences and present a construction of Witt-correspondences (that produce it from rational function on relative curve and trivialisation of the fraction of canonical classes of curve and base variety) and show that derived Witt-groups are presheaves with Witt-transfers. 
In the second section we prove 
mentioned rigidity theorem.
In Appendix we recall the prove of rigidity for Picard group based on proper base change according the reason from the proof of theorem from \cite{SV_SingHomofSchemes}. 

Author thanks to I.~Panin who encouraged him to deal with this question,
and to A.~Ananievskiy for useful consultations and remarks.

\section{Witt-correspondences}
In this section we give definition of the category $WCor_k$ (involved in \cite{AD_WittCor} ) that is category of Witt-correspondences and proof some lemmas about Witt-correspondences used in the next section. 
In purposes of the article it is enough to define Witt-correspondences between smooth affine varieties.
In both following sections we assume that characteristic of base filed $k$ isn't equal to $2$ 
($char\,k\neq 2$). 

To give the definition of Witt-correspondences we define following exact category with duality.

\begin{definition}
Let $X$ be smooth affine variety and $s\colon S\to X$ be affine scheme over $X$.
Let denote by
$\mathcal P(S_X)$ (or by $\mathcal P(s)$) 
is full subcategory of the category $k[S]$-modules 
  consisting of modules $M$ that are finitely generated projective over $k[X]$.
The internal $\mathcal Hom$-functor represented by $k[X]$ 
defines duality functor on $\mathcal P(S_X)$, i.e.
$$M\mapsto D_X(M)= Hom_{k[X]}(M, k[X] )$$
where right side is considered as $k[S]$-module according to the action of $k[S]$ on $M$
(it is duality since it 
commutes with restriction of scalars along $k[X]\to k[S]$ and 
it defines duality on the category of finitely generated projective modules over $k[X]$)

So we get exact category with duality $(\mathcal P(S_X),D_X)$.
In particular, 
for affine smooth varieties $X$ and $Y$
we denote by $(\mathcal P(X,Y), D_X)$ (or $(\mathcal P_X^Y, D_X)$) 
the category $(\mathcal P(X\times Y\to X), D_X)$
. 
\end{definition}

\begin{proposition}\label{pro_compQSCor}
There is a functor between categories with duality
$$-\circ -\colon (\mathcal P^Y_X,D_X) \times (\mathcal P^Z_Y,D_Y) \to (\mathcal P^Z_X,D_X)$$
defined for all affine smooth varieties $X$, $Y$ and $Z$ and it is natural along them
(product of categories with duality 
  $(\mathcal P^Y_X,D_X) \times (\mathcal P^Z_Y,D_Y)$ is considered as  
  category with duality $D_X\times D_Y$
).
\end{proposition}
\begin{proof}
To define functor of categories with duality we should define functor 
  $$\begin{aligned} 
  \mathcal P^Z_Y& \times& \mathcal P^Y_X &\to \mathcal P^Z_X \\
  (N&,& M)& \to N\circ M\end{aligned}$$ and 
natural isomorphism between its compositions with duality functors at the left and right
  $\nu\colon D_X(N\circ M) \simeq (D_Y(N) \circ D_X(M) )$.

We define required functor as tensor product over $k[Y]$, i.e.
consider functor 
$$\begin{aligned}
k[Y\times Z]-Mod&\times& k[X\times Y]-Mod &\to k[X\times Z]-Mod \\
(N&,& M) &\mapsto  M\otimes_{k[Y]} N
\end{aligned}$$
and show that it defines functor from $\mathcal P^Y_X\times\mathcal P^Z_Y$ to $\mathcal P^Z_X$.
By definition these categories consists of modules $M$ over $k[X\times Y]$ ($N$ over $k[Y\times Z]$ and $K$ over $k[X\times Z]$) 
that are finitely generated projective over $k[X]$ ($k[Y]$ and $k[X]$ respectively).
So we should show that    
  if $M$ is finitely generated projective over $k[X]$ and $N$ -- over $k[Y]$ 
  then $M\otimes_{k[Y]} N$ is finitely generated projective over $k[X]$ 
It follows from that 
tensor product  preserves direct sums. In fact,  
if $N$ is free module of rank $n$ over $k[Y]$ then $C(M\times N)\simeq M^n$, so
if $M_{k[X]}$ is projective and $N_{k[Y]}$ is free then tensor product $M\otimes N$ is projective over $k[X]$,
and so by additivity again the same holds if $N_{k[Y]}$ is projective.

Then we define natural transformation $\nu$ as composition of the following isomorphisms
\begin{multline}
D_X(M\otimes_{k[Y]} N)=Hom_{k[X]}(M\otimes_{k[Y]} N, k[X])\simeq\\
\simeq
Hom_{k[Y]}(N, Hom_{k[X]}(M, k[X]) )\simeq
 Hom_{k[Y]}(N,k[Y])\otimes Hom_{k[X]}(M, k[X])=\\=D_X(M)\otimes_{k[Y]} D_Y(N) \label{DualTensProd}
 \end{multline}
Here the second isomorphism is adjunction isomorphism of tensor product and internal hom-functors,
and the third one follows from that $N$ is projective over $k[Y]$

\end{proof}

\begin{definition}
Category $WCor_k$ is additive category  
  that objects are smooth affine varieties over $k$, 
  morphisms groups  are defined as $WCor(X,Y) = W(\mathcal P^Y_X,D_X)$
    where $W$ is Witt-group of exact category with duality (according to Balmer definition \cite{Bal_DerWittgr}),
  and compositions is induced by functor $-\circ-$ from proposition \ref{pro_compQSCor},
  identity morphism $Id_X$ is defined by class of quadratic space $(k[\Delta],1)$ 
    where $\Delta$ is diagonal in $X\times X$ and free of rank one 
     $k[X]$-module is equipped with unit quadratic form.
     
 Presheave with Witt-transfers is an additive presheave $F\colon WCor_k\to Ab$.     
\end{definition}

\begin{proposition}\label{prop_WittgrWitttr}
The derived Witt-groups are presheaves with Witt-transfers,
or more strictly it means that
for any $i\in \mathbb{Z}/4\mathbb{Z}$
there is a functor $\bF\colon WCor\to Ab$ 
such that 
$\bF(X)=W^i(X)$ for any smooth variety $X$
and $\bF(f)=W^i(f)$ for any regular morphism.
The presheaves $W^i(X\times -)$ for smooth affine $X$ are presheaves with Witt-transfers in the same sense. 
\end{proposition}
\begin{proof}
The claim follows from proposition \ref{pro_compQSCor}
since $W^i(U)= W^i(\mathcal P^{pt_k}_U, D_U)$ for any smooth affine $U$
and from that multiplication by smooth affine $X$ induces functor $X\times - \colon WCor_k\to WCor_k$
\end{proof}

\begin{remark}
Functor $(X,Y)\mapsto (\mathcal P^Y_X,D_X)$ with natural transformation between functors 
  $(X,Y,Z)\mapsto (\mathcal P^Y_X,D_X)\times (\mathcal P^Y_X,D_X)$ and $(X,Y,Z)\mapsto (\mathcal P^Z_X,D_X)$
  defined by the functor $C$ form proposition \ref{pro_compQSCor}
defines category $\mathcal{Q}Cor$ that objects are smooth varieties and that is 
  enriched category  over the category of categories with duality. 

Then the category $WCor$ is result of applying the functor fo Witt-groups to $\mathcal QCor$.
 And it is the category in zero degree of $\mathbb Z/4$-graded category that is result of applying whole Witt cohomology theory. 
\end{remark}

\begin{definition}
$$\overline{WCor}(X,Y) = coker ( WCor(X\times\affl,Y) \xrightarrow{r_0^*-r_1^*} WCor(X,Y) )$$
where $r_0,r_1\colon X\to X\times\affl$ denotes unit and zero sections.
\end{definition}
\begin{definition}
Homotopy invariant presheave with Witt-transfers $\bF$ on the category of affine smooth schemes is a presheave with Witt-transfers that is homotopy invariant, i.e. such that $\bF(X\times\affl)\simeq\bF(X)$ for any smooth affine $X$,
and this is the same to say that $\bF$ is presheaf on the category $\overline{WCor}$.
\end{definition}

\begin{denotation}\label{def_twaut}

For any variety $U$ and invertible function $q\in k[U]^*$ we denote by 
$\langle q\rangle\in WCor(U,U)$ an endomorphism of $U$ defined by
class of quadratic space $(k[\Delta],(q) )$
where
  $\Delta$ denotes diagonal in $U\times U$
  and $(q)$ is $1\times 1$ matrix defining
  quadratic $k[U\times pt]$-form on $k[U\times U]$-module $k[\Delta]$
  since it is free module of rank 1 over $k[U\times pt]$. 

Note that all these endomorphisms are automorphisms and that under this denotation $\langle 1\rangle \in WCor(U,U)$ is identity morphism of $U$.
\end{denotation} 

\begin{remark}\label{rm_twistmap}
Any regular map of smooth affine varieties $\varphi\colon U\to V$ can be considered as Witt-correspondence that is the class
of the space $(k[\Gamma], 1)$, where $\Gamma\subset X\times Y$ is graph of $\varphi$.
We will denote this morphism by $[\varphi]$.

So there is a functor $Sm_k\to WCor_k$.
Moreover $\Phi\circ \varphi= \varphi^*(\Phi)$ where $\varphi^*\colon WCor(Y,Z)\to WCor(X,Z)$ is homomorphism is induced by the scalar extension functor $\mathcal P(Y,Z)\to \mathcal P(X,Z)$. 

In addition since any quadratic form on the free module of rank one is defined by invertible function, any Witt-correspondence that is defined by quadratic space with module $k[\Gamma]$ is equal to $[\varphi]\circ\langle q\rangle$ for some $q\in k[U]^*$. 
\end{remark}

\begin{lemma-remark}\label{lm_unitclp-symb=unit}
For any local henselian $U$ with closed point $u$ and invertible function $q\in k[U]^*$ 
$$q(u)=r^2, r\in k(u)* \implies \langle q,U\rangle = id_U\in WCor(U,U).$$ 
\end{lemma-remark}
\begin{proof}
Since $U$ is henselian local,  $q(u)$ is square and $char\, k\neq 2$ then $q$ is square, i.e. $q=w^2$ for some invertible $w\in k[U]$.
Then homomorphism of multiplication by $w$ in $k[\Delta]$ ($\Delta\subset U\times U$ is diagonal) 
induces isomorphism $(k[\Delta],(q) )\stackrel{-\cdot w}{\simeq} (k[\Delta],(1) )$. Thus $\langle q\rangle = \langle 1\rangle\in WCor(U,U)$.  
\end{proof}

Now we give construction that produce Witt-correspondence $U\to Y$ from "curve-correspondence" $U\leftarrow \bC\to Y$ with "orientation" i.e. trivialisation of relative canonical class and a "good" function on $\bC$. 
\begin{definition}\label{rm_relcurve}
We call by smooth projective (affine) relative curve $\mathcal C$ over scheme $U$ a projective (affine) morphism $\mathcal C\to U$ of relative dimension 1.  
\end{definition}
\begin{denotation}
We denote by $Z(s)$ and $Z(f)$ vanish locus of section $s$ of some line bundle or vanish locus of the regular function $f$. 
\end{denotation}

\begin{definition}\label{def_OrCurF}
Let $U$ be a smooth affine variety over $k$, 
$\ovbC$ be a relative projective curve over $U$ (see definition  \ref{rm_relcurve}),
$\bC$ is open subscheme in $\ovbC$ such that $\bC$ is smooth over $k$ and affine,
and complement $\bCinf=\ovbC \setminus \bC$ is finite over $U$.
Let $\mu\colon \omega(\bC)\otimes \omega(U)^{-1}\simeq\bO(\bC)$ be some trivialisation, 
let $f\colon \ovbC\to\prl$ be a regular map
that isn't constant at each irreducible component of each fibre of $\ovbC$
and moreover such that 
$Z(f) = f^{-1}(0)$ is contained in $\bC$.
Finally let 
$Z$ be isolated component of $Z(f)$ (i.e. $Z(f)=Z\coprod Z^\prime$ for some $Z^\prime$)
and let $g\colon Z\to Y$ be a regular map for some smooth affine $Y$
$$\xymatrix{
Z\ar[r]\ar[dd]_g & Z\coprod Z^\prime \ar@{^(->}[d] \ar[rr]  && 0\ar@{^(->}[d]\\
& \bC\ar@{^(->}[r] \ar[dr]&\ovbC\ar[r]^f\ar[d]&\prl \\
Y& & U &  
}
$$

Then we denote by $OrCurF(U)$ the set of all such data $(U,\ovbC,\bC, \mu, f ,Z)$ (without $g$ and $Y$) for fixed smooth affine variety $U$,
and denote by $OrCurF(U,Y)$ the set of all such data $(U,\ovbC,\bC, \mu, f ,Z , g)$ for fixed smooth affine $U$ and $Y$
$$
OrCurF(U) \stackrel{def}{=} \{ (U,\ovbC,\bC, \mu, f ,Z) \},
OrCurF(U,Y) \stackrel{def}{=} \{ (U,\ovbC,\bC, \mu, f ,Z, g) \}
.
$$
\end{definition}

\begin{construction}\label{const_OrWLf-WCor}

For any $(U,\ovbC,\bC,\mu,f , Z)\in OrCurF(U)$ we construct a quadratic space $\Phi=(k[Z],q)$ in $\mathcal P(Z_U)$,
i.e. $k[Z]$-linear isomorphism $q\colon Hom_{k[U]}(k[Z],k[U])\simeq k[Z]$:
$$ (U,\ovbC,\bC,\mu,f , Z)\mapsto \Phi=(k[Z],q) $$.

To do this firstly we consider 
  the morphism of relative projective curves 
  $\Pi_{proj}=(f,pr_U)\colon \ovbC\to \prl_U$
  that is finite morphism of schemes since it is projective and quasi-finite (since $f$ is not constant on each fibre).
Then we consider its restriction to the morphism 
\begin{gather*}
\Pi_{triv}\colon \bCp\to V\colon  \\
\begin{aligned}
V &= \affl_U-\Pi_{proj}(\bCinf) \,(= \prl_U-(\infty_U\cup \Pi_{proj}(\bCinf))\, ) \\
\bCp &= \Pi_{proj}^{-1}(V) = \ovbC- ( \Pi_{proj}^{-1}( \Pi_{proj}(\bCinf)) \cup f^{-1}(\infty) ) .
\end{aligned}\end{gather*}
Note that since $Z(f)=f^{-1}(0)\subset \ovbC-\bCinf$ then $0_U\in V$ and $Z(f) \subset \bCp$.
The morphism $\Pi_{triv}$ is finite morphism of smooth affine schemes 
(Note: 
affinness of $\bC$ is not necessary to be required in the starting assumptions in fact 
since $\ovbC-f^{-1}(\infty)$ is finite over affine scheme $\affl_U$). 
So $\Pi_{triv}$ is flat (see \cite{AK_GrothDual}) 
and by proposition 2.1 form \cite{OP_Witt} there is an isomorphism 
$${q^\omega_{dif}}^\prime\colon 
   Hom_{k[\affl_U]}(k[\bC^\prime ],k[V] )\simeq 
   \omega(\bC^\prime )\otimes \Pi_{triv}^*(\omega(V))^{-1}$$
Multiplying the last isomorphism by 
  trivialisation $\mu$ (i.e. its restriction on $\bCp$)
  and trivialisation $dT^{-1}$
  of canonical class $V$ over $U$
we get isomorphism
\begin{equation*}  q_{dif} = q_{def}^\omega\otimes\mu\otimes dT^{-1}\colon
   Hom_{k[V]}( k[\bCp ] , k[V ])\simeq 
   k[ \bCp ]     
\end{equation*}

Now by base change along the embedding 
$U\times 0\hookrightarrow V$ we get $k[Z]$-linear isomorphism 
$$q_0\colon Hom_{k[U]}(k[Z(f)],k[U])\simeq k[U] $$
since 
the fibre of $\Pi_{triv}$ over $U\times 0$ is $Z(f)$.
Finally since $Z(f)= Z\coprod Z^\prime$ and since $q$ is $k[Z(f)]$-linear 
there is splitting of quadratic space $(k[Z(f)],q)=(k[Z],q)\oplus (k[Z^\prime],q^\prime)$. 
\end{construction}

\begin{remark}
In the case of $\ovbC=\prl_U$, $\bC=\affl_U$, $f$ being polymon in $k[U][t]=k[\affl_U]$ and trivialisation $\mu$ defined by differential of coordinate on affine space
the quadratic form $q_{dif}$ in previous construction is equal to the quadratic form defined with Euler trace, i.e for $a,b\in k[\affl_U]=k[U][t]$ $\langle a,b\rangle=tr(ab/f^\prime)$ where $f^\prime$ is derivative of $f$.
\end{remark}

\begin{definition}
For any smooth affine $U$, $Y$ and any $(U,\ovbC,\bC, \mu, f ,Z , g)\in OrCurF(U,Y)$ let's denote by 
$\langle \ovbC,\bC, \mu, f ,Z , g \rangle$
the element in $WCor(U,Y)$ defined by class of quadratic space constructed by \ref{const_OrWLf-WCor}.
Thus we get map
\begin{equation*}\begin{aligned}
OrCur(U,Y)\to &WCor(U,Y)\\
( \ovbC,\bC, \mu, f ,Z , g) \mapsto & [\Phi]=\langle \ovbC,\bC, \mu, f ,Z , g \rangle
\end{aligned}\end{equation*}

\end{definition}

\begin{lemma}\label{lm_SqMet}
1) 
Let $U$ and $Y$ are smooth varieties and dimension of all components of $Y$ is bigger then $0$. 
Let $(P,q)$ be quadratic space in $\mathcal P^Y_X$ 
(let's recall that it means that $q$ is $k[Y\times U]$-linear isomorphism $P\simeq Hom_{k[U]}(P,k[U])$)
 such that 
$P\simeq k[W]/(e^n)$ (as $k[Y\times U])$-module) for some open $W\subset Y\times U$ and regular function $e\in k[W]$.
Then
$$\begin{cases}
(P,q) \text{ is metabolic } \text{ , if } n \text{ is even} \\
(P,q)\simeq (P^\prime,q^\prime) \text{ for } P^\prime\simeq k[W]/(e) \text{ (as }k[Y\times U]-{module ) } \text{ , if } n \text{ is odd}
\end{cases}$$

2)
For any $f=t^n\in k[t]=k[\affl]$ for odd $n$ an endomorphism of $pt_k$ in $WCor_k$ defined by $f$, i.e. $\langle pt_k,\affl_k ,dt, f, Z(f), \affl_k\to pt_k \rangle\in WCor(pt_k,pt_k)$ is automorphism.
\end{lemma}
\begin{proof}
1)
Let's check that for any $k[W]$-linear quadratic form $$q\colon Hom_{k[U]}(k[Z(e^n)],k[U]\simeq k[Z(e^n)]$$ ($Z(e^n)=Spec\, k[W]/(e^n)$) for all $i=0\dots n$ 
$(e^i)^\bot = (e^{n-i})^\bot$ where ideal $(e^i)$ is considered as subspace in $(k[Z(e^n)],q)$ and $-^\bot$ denotes orthogonal in respect to $q$.
Indeed,
\begin{multline*}
a\in (e^i)^\bot \Leftrightarrow
\langle a,e^i\cdot b\rangle_q=0\;(\forall b\in k[W]) \Leftrightarrow \\
\langle a\cdot e^i,b\rangle_q=0\; (\forall b\in k[W]) \Leftrightarrow
a\cdot e^i \in (e^n) \Leftrightarrow
a\in (e^{n-i})
\end{multline*} 
where 
the first equivalence is definition of orthogonal complement. 
The second one follows from that $q$ is linear over $k[W]$.  
The third follows from that $q$ is isomorphism and so kernel of pairing on $k[Z(e^n)]$ is zero 
(and kernel of the pairing on $k[W] $ is $(e^n) $). 
The last equality follows from that multiplication by $e$ in $k[W]$ is injective. 
(If it isn't injective, then $e$ is zero on some irreducible component of $W$. Dimensions of irreducible components or $W$ are the same as for $U\times Y$, and so dimension for any such component is bigger then $dim\, U$ by assumption on dimension of irreducuble components of $Y$, but by definition of $\mathcal P^Y_U$ $P\simeq k[Z(e^n)]$ is finite over $k[U]$ and so $Z(e^n)$ and $Z(e)$ are finite over $U$)

Thus if $n=2l$ then $(e^l)$ is lagrangian subspace in $(P,q)$ so $(P,q)$ is metabolic.
If $n=2l+1$ then $(e^{l+1})$ is sublagrangian with orthogonal $(e^{l})$ and by sublagrangian reduction $(P,q)\simeq ( (e^{l})/(e^{l+1}), q^\prime)$ 
and $(e^{l})/(e^{l+1})\simeq k[W]/(e)$

2)
The second claim follows immediately since $k[t]/(t)\simeq k$ and by first point $(k[t]/(t^n),q)\simeq (k[t]/(t),q^\prime) = (k,\lambda)$,$\lambda\in k^*$ so it is automorphism.
\end{proof}

\section{Rigidity theorem}
Let's recall that in this section we assume 
$char\,k\neq 2$.

\begin{theorem}\label{th_WCorSectRidg}
Let 
  $\ovbC$ be  a relative projective curve (see definition \ref{rm_relcurve}) over scheme $U$ 
    that is henselisation of smooth local scheme at closed point with separable residue filed over base field $k$.
Let $\bCinf\subset \ovbC$ be closed subsceme that is finite over $U$ and such that open complement $\ovbC-\bCinf$ is smooth over $U$.
Then for any two sections 
  $S_0,S_1\colon U\to \ovbC$ 
    coinciding in closed fibre, i.e. such that $S_0(u) = S_1(u)=x\in \ovC$ 
    (where $u\in U$ denotes closed point and $C\subset \bC$ denotes closed fibre of $\bC$)
$$\xymatrix{
\bC & C\ar@{^(->}[l]  &\\
U\ar@<-1ex>[u]_{S_1}\ar@<1ex>[u]^{S_0} & u\ar@{^(->}[l]\ar[u]^x
&,}$$
homotopy classes of sections $S_0$ and $S_1$ i.e. corresponding morphisms in the group $\overline{WCor}(U,\bC)$
are the same, i.e.
$$S_0,S_1\colon U\to \bC, S_0(u)=S_1(u)\implies [S_0] = [S_1] \in \overline{WCor}(U,\bC).$$
\end{theorem}

\begin{corollary}\label{cor_RigidPresh}
Let $X$ is algebraic variety over a field $k$ with $char\,k\neq 2$ and $x$ is smooth closed point with separable residue field $k(x)/k$.
Then for any homotopy invariant presheave $\bF$ with Witt-transfers
$F(Spec\,\mathcal{O}^h_{X,x})\simeq \bF(x)$

\end{corollary}
\begin{proof}[Proof of the corollary \ref{cor_RigidPresh} ]

{\em case of infinite base field)}

Reduction of the statement to the theorem \ref{th_WCorSectRidg} is by standard argument (see \cite{SV_SingHomofSchemes}, proof of the theorem 4.4)  for example. Nevertheless let's present it.

Consider base change form $k$ to $k(x)$, since $k(x)$ is separable then diagonal $\Delta_x\subset x\times x$ is isolated component and since $X_{k(x)}\to X$ is etale then $(X_{k(x)}, \Delta_x)\to (X,x)$ is Nisnevich neighbourhood and then henselisation of $X_{k(x)}$ at $\Delta_x$ is equal to henselisation of $X$ at $x$ so we can assume that point $x$ is rational. 
 
Then let's note that for any variety $X$ of dimension $d$ with rational closed point $x$ henseliastion of $X$ is equal to henselization of $\aff^d$ at $0$ (because there exist embedding of $X$ in affine space and projection to target space of $X$ is separable at $x$ ).
So we can assume that $(X,x)$ is equal to $(\aff^d,0)$ and let denote $S_d=Spec\,\mathcal O^h_{\aff^d,0}$.

Now we prove that inclusion and projection morphisms between $S_d$ and $S_{d-1}$ induces isomorphisms 
$\bF(S_d)\simeq \bF(S_{d-1})$
 and then by induction the claim follows.
The composition $S_{d-1}\to S_d\to S_{d-1}$ is identity.
To show that the second composition induce the identity on $\bF(S_d)$ we show that for any Nisnevich neighbourhood $e\colon (U,u)\to (\aff^d,0)$ there is a homotopy in $WCor$ between the canonical morphism $i\colon S_d\to U$ 
and composition 
$r\colon S_d\to S_{d-1}\hookrightarrow S_d\to U$ (considered as Witt-correspondeces).

To reduce the last statement to the question on sections of relative curve over $S_d$ we use 
Quillen trick and firstly let's show how to construct finite morphism $f\colon U\to \aff^d_k$.

Consider 
any embedding of $U$ as proper closed subscheme in some affine space $j\colon U\hookrightarrow\aff^N_k$.
We get finite morphism $f$ as projection in the center at some subspace of dimension $N-d-1$ in infinity subspace $\pro^{N-1}_k$   and we find this subspace by consequently choosing rational $N-d$ points on infinity, or in other words we consequently consider projections in projective spaces with center at points $p_N,p_{N-1},\dots p_{d+1}$, $p_i$ laying in infinity subspace of $i$-dimensional projective space, such that each $p_i$ don't belong to the image of the projective closure of $U$ under the previous projection. The next paragraph explains this precisely. 

Namely, consider affine space as open subscheme in projective space $\aff^N_k\subset\pro^N_k$,
And let $\overline{U} = Cl_{\pro^n}(j(U))$, $B=\overline U\cap \pro^{n-1}$ be closure of the image of $U$ in projective space  and its intersection with inifinite projective subspace $\pro^{N-1}_k$ (i.e. complement $\pro^N_k\setminus \aff^N_k$).
Since $U$ is proper closed subscheme in $\aff^N$ then $d<N$ 
and since $dim\,B<dim\,\overline{U}=dim\,U=d$, then
$B$ is proper closed subscheme in $\pro^{N-1}_k$.
If $k$ is infinite then there is a closed rational point $p_N\in \pro^{N-1}_k$, $p_N\not\in B$ and let define morphism $f_{N-1}\colon \overline{U}\to \pro^{N-1}_k$ as projection with center $p_N$.
That fact that $p_N$ does not lays in $B$ and consequently in $\overline{U}$ implies that $f_{N-1}$ is finite morphism.
(Indeed projection in the center $p_{N}$ is projective morphism $pr_N\colon Bl_{p_N}(\pro^N_k)\to \pro^{N-1}_k$ and since $p_N\not\in \overline{U}$ preimage of $\overline{U}$ is $Bl_{p_N}(\pro^N_k)$ is isomorphic to $\overline{U}$, so we get well defined projective morphism $\overline{U}\to \pro^{N-1}_k$. Again since $p_N\not\in \overline{U}$ and so preimage of $\overline{U}$ doesn't intersect with exceptional divisor in $Bl_{p_N}(\pro^N_k)$, then preimage of $\overline{U}$ doesn't contain any fibre of the projection $pr_N$. And since this fibres are projective lines, and hence $f_{N-1}$ is quasi-finite morphism, then it is finite.)
So we get finite morphism $f_{N-1}\colon \overline{U}\to \pro^{N-1}_k$ and since $p_N$ lays in infinite subspace of $\pro^N_k$, then $f_{N-1}$ sends $U$ to affine subspace $\aff^{N-1}_k$ and $B$ -- to infinite subspace $\pro^{N-2}=\pro^{N-1}\setminus \aff^{N-1}_k$. 
If $d<N-1$ then again there is a point $p_{N-1}\in\pro^{N-2}_k$, $p_{N-1}\not\in f_{N-1}(B)$, and projection at the center $p_{N-1}$ give us finite morphism $f_{N-2}\colon \overline{U}\to \pro^{N-2}_k$ that sends $U$ to $\aff^{N-1}_k$ and $B$ to infinite subspace. So consequently choosing center of projections, finally we get finite morphism $f_d\colon \overline{U}\to \pro^d_k$  and $f\colon U\to \aff^N_k$.

Now consider any linear projection $\aff^d_k\to \aff^{d-1}_k$ and base change of $f$ along composition map $S_d\to \aff^d_k\to\aff^{d-1}$ that is finite morphism of relative curves $\bC\to \affl_{S_d}$ over $S_d$
and next consider normalisation $\bC\hookrightarrow \ovbC \to \prl_{S_d}$ of composition morphism $\bC\to \affl_{S_d}\hookrightarrow \prl_{S_d}$.
$$
\xymatrix{
\bC\ar[r]^v \ar[d] \ar[rd]|>>>>>>v & \affl_{S_d}\ar@/_/[ld] \ar[rd] \\
S_d\ar[rd] & U\ar[r]^f \ar[d] & \aff^d_k\ar[ld]\\
 &\aff^{d-1}_k
}
\quad\quad
\xymatrix{
\ovbC\ar[r] & \prl_{S_d} \\
\bC\ar[r] \ar[d]\ar@{^(->}[u]& \affl_{S_d}\ar[d]\ar@{^(->}[u]\\
S_d \ar@{=}[r] & S_d
}  $$
Since $\overline{f_{S_d}}$ is finite inverse image of $\bO(1)$ is ample invertible sheave on $\ovbC$.
Morphisms $i$ and $r$ from $S_d$ to $U$ define sections $g_0,g_1\colon S_d\to\bC$ that coincide at closed fibre,
and applying theorem \ref{th_WCorSectRidg} to relative curves $\ovbC$, $\bC$ and ample sheave $f_{S_d}^*(\bO(1))$ we get that
$[g_0]=[g_1]\in WCor(S_d,\bC)$ and hence $[i]=v \circ [g_0]=v\circ [g_1]=[r]\in WCor(S_d,U)$

{\em (case of finite base field)}
In previous reasoning we use infiniteness of base field $k$ in choices of rational points $p_{n}\in\pro^{n}_k$, $p_n\not\in f_n(B)$.  
In the case of finite base field $k$ we consider finite odd extensions $K/k$ and base change of all previous constructions to $K$. 
Since $deg\,f_n(B)=deg\, B$ is fixed (and since projective space over the finite field $K$ has at least $\#K$ rational points (where $\#K$ is the number of elements in $K$) ), for all extensions $K/k$ with $deg\,K/k$ bigger some $M$ we can find $K$-rational points $p_n\in \pro^n_K$, $p_n\not\in f_n(B)$, and so theorem \ref{th_WCorSectRidg} yield the statement of corollary \ref{cor_RigidPresh} for presheaves on $Sm_K$ for all extensions $K/k$ with $deg\,K/k>M$. Now let's choose some primitive separable extension $K/k$ of odd degree bigger then $M$.

Let $p=t^n+a_{n-1}t^{n-1}+\dots a_0$ be irreducible polynom defining extension $K/k$.
Then $Z(p)\subset \affl_k$ is isomorphic to $Spec\, K$, so construction \ref{const_OrWLf-WCor} gives us morphism 
$\varepsilon = \langle p\rangle \in WCor(pt_k,pt_K)$.
(Namely this morphism is defined by quadratic form $tr(-\cdot -)$ on $K$ (over $k$) where $tr\colon K\to k$ is the trace and it is equal to base change along closed embedding $Spec\,k=0\hookrightarrow \affl_k$ of the quadratic form on $k[t]$ defied by Euler trace .)

Denote $pt_k=Spec\,k$, $pt_K=Spec\,K$. 
Consider diagram 
$$\xymatrix{
U\times pt_K\ar[r]^-{p_U} & U\\
S\times pt_K\ar@<1ex>[u]^{{S_0}_K}\ar[u]_{{S_1}_K}\ar@<1ex>[r]^-{p} & S\ar@<1ex>[u]^{S_0}\ar@<-1ex>[u]_{S_1}\ar@<1ex>[l]^-{\epsilon}
}$$
Now let's show that composition $p\circ \varepsilon\colon WCor(pt_k,pt_k)$ has the same homotopy class as some automorphism $\langle\lambda\rangle\in WCor(pt_k,pt_k)$. 
In fact polynom $h= p\cdot(1-\lambda)+t^n\cdot \lambda$
defines homotopy $\langle h , \affl\times {\affl_k}\to pt_k \rangle\in WCor(\affl_k,pt_k)$
between $p\circ \varepsilon = \langle p , \affl_k\to pt_k \rangle$ and $\langle t^n, \affl_k\to pt_k\rangle$  in $WCor(pt_k,pt_k)$.
And by lemma \ref{lm_SqMet} $\langle t^n, \affl_k\to pt_k\rangle = \langle \lambda \rangle$ for some $\lambda\in k^*$. 

So 
$$0 = p_U \circ [{S_0}_K]-[{S_1}_K]\circ \varepsilon =[{S_0}]-[{S_1}] \circ p\circ\varepsilon = [{S_0}]-[{S_1}] \circ \langle\lambda\rangle,$$ and $[{S_0}]-[{S_1}]=0$.

\end{proof}
\begin{corollary}
For a smooth variety $X$ over field $k$ , $char k\neq 2$ and for a closed point $x$ on $X$ with separable residue field $k(x)/k$
and any smooth variety $Y$
$$W^i(Y\times \mathcal Spec\,O^h_{X,x})\simeq W^i(Y\times Spec\,k(x))$$
where $W^i$ $(i\in \mathbb Z/4\mathbb Z)$ derived(higher) Witt groups.
\end{corollary}
\begin{proof}
Since presheaves $W^i(Y\times - )$ for smooth affine $Y$ are presheaves with Witt-transfers by proposition \ref{prop_WittgrWitttr} 
and since this presheaves are homotopy invariant (by homotopy invariance of the presheaves $W^i(-)$)
previous corollary implies the statement for smooth affine $Y$.
The statement for general smooth $Y$ follows by Mayer-Vietoris property 
\end{proof}

Before we give proof of the theorem \ref{th_WCorSectRidg} firstly let's briefly explain the idea.
{\par\em Idea of the proof of theorem \ref{th_WCorSectRidg}.}
We want to prove that regular maps $S_0,S_1\colon U\to \bC$ coinciding in closed fibre define the same 
classes up to $\affl$-homotopies of Witt-correspondences. 
As was shown in construction \ref{const_OrWLf-WCor} Witt-correspondence 
can be defined by trivialisation of canonical class of $\bC$ (that exist after shrinking of $\bC$) and rational function on the compactification $\ovbC$.

Rational function can be presented as fraction of pair of global sections $s,t$ of line bundle $\bL$ on $\ovbC$ or after shrinking again of $\bC$ as a global section and (good)trivialization of this line bundle on $\bC$.
Under such correspondence $WCor$-morphisms defined by regular maps $S_0,S_1$    
relate to sections $s_0\in \Gamma(\ovbC,\bL(S_0))$, $div\,s_0=S_0$ , $s_1\in \Gamma(\ovbC,\bL(S_1))$, $div\,s_1=S_1$, and since closed fibre of $S_0$ and $S_1$ are the same then $s_0$ and $s_1$ can be chosen in such way that their closed fibres coincide.  
The mentioned correspondence between $WCor(U,\bC)$ and sections on line bundles on $\ovbC$ is (partial) additive in some sense and if we continue it to rational sections then difference $[S_0]-[S_1]$ corresponds to fraction $s_0/s_1$. 

$\affl$-homotopies don't change line bundle and if we don't change its trivialisation on $\bC$ too then class up to $\affl$-homotopy of the Witt-correspondence is defined by restriction of the section on the complement $\ovbC\setminus\bC$.
So class of $[S_0]-[S_1]$ in $\overline{WCor}(U,\bC)$  is defined by $(s_0/s_1)\big|_{\ovbC-\bC}$.

By rigidity for Picard group (and since $ckar\, k\neq 2$) $\bL(S_0)\otimes \bL(S_1)^{-1}$ is square.
Scheme $\ovbC\setminus\bC$ is finite over henselian scheme $U$ (we can make all mentioned shrinks of $\bC$ in such that complement remains to be finite),
hence it is henselian and since $s_0$ and $s_1$ coincide at closed fibre then $(s_0/s_1)\big|_{\ovbC-\bC}$ is square.
But square of section define zero morphism in $WCor$ by lemma \ref{lm_SqMet} (here we mean some lift of square root of $(s_0/s_1)\big|_{\ovbC-\bC}$ to rational section $e/d$ on $\ovbC$).

\vspace{5pt}

In the following proof we avoid using of rational sections and implicit trivialisations of line bundles. So homotopy connecting $[S_0]$ with $[S_1]$ is defined by Witt-correspondence constructed from the map $$[((1-t)\cdot s_0\cdot d^2+ t\cdot s_1\cdot e^2) \colon r]\colon \ovbC\to \prl ,$$ where $e\in \Gamma(\ovbC,\bL(l))$, $\bL=\bL(S_0)\otimes\bL(S_1)^{-1}$, $d\in\Gamma(\ovbC,\bO(l))$,  $r\in \Gamma(\ovbC,\bL(S_0)\otimes\bO(l))$, and consequently choosing these and some additional auxiliary sections we accurate control that vanish locus don't intersect.
  
\qed

\begin{proof}{Proof of the theorem \ref{th_WCorSectRidg}}

During the proof we will need to shrink $\bC$ (and extend $\bCinf$), so
it will be useful to involve the following definition-denotation.
\begin{definition}\label{def_subcomcurve}
Let's call by compactified curve over $U$ 
  a set $(\bC,\ovbC,\bCinf,\bO(1))$ 
    where $\ovbC$ is smooth projective relative curve over base $U$, 
    $\bCinf\subset \ovbC$ is closed subscheme finite over $U$, 
    $\bC$ is its complement i.e. $\bC = \ovbC-\bCinf$, 
    and $\bO(1)$ is any ample bundle on $\ovbC$. 
In short we'll call that $\bC = \ovbC-\bCinf$ is compactified curve 
  omitting ample bundle if it can be recovered from the context .

Let's say that compactified curve $(\sbC,\ovbC,\sbCinf,\bO(1))$ 
  is subcurve of compactified curve $(\sbC,\ovbC,\bCinf,\bO(1)^\prime)$ 
  if $\sbC\subset \bC$ and $\bO(1)^\prime=\bO(k)$ for some $k> 0$.
  
\end{definition}

Now we prove some lemmas that allow to equip relative curve $\ovbC$ with
  linear bundles and its sections that will be used in the construction of $Witt$-correspondence.
We divide construction of this data 
  in three steps-lemmas, 
and note that in formulations of these lemmas 
  denotations $\bC$, $\bCinf$ and $\bO(1)$ are used as free terms, 
  i.e. it shouldn't necessarily be equal to $\bC$ and $\bCinf$ from formulation of the theorem 
  (and in fact the third lemma will be applied to a subcurve $(\sbC,\ovbC,\sbCinf,\bO(1)^\prime)$ constructed in the second lemma).

\begin{lemma}\label{lm_s0s1zeta}
For any compactified curve $(\bC,\ovbC,\bCinf,\bO(1))$ 
with sections $S_0,S_1\colon U\to\bC$ coinciding at closed point $S_0(u)=S_1(u)=x$
there are linear bundles $\bL_0$, $\bL_1$, $\bL$
with following sections and isomorphisms
\begin{equation}\label{eq-sectcond1}\begin{aligned}
s_0\in \Gamma(\ovbC,\bL_0),  Z(s_0)=S_0,  \quad
s_1\in\Gamma(\ovbC,\bL_1),  Z(s_1)=S_1,\\
\zeta\in\Gamma(\ovC,\bL),\;
\tau\colon \bL_0\simeq\bL_1\otimes\bL^2,\;
 \tau(s_0)\big|_{\ovC}=s_1\big|_{\ovC}\cdot \zeta^2 
\end{aligned}\end{equation}
\end{lemma}

\begin{lemma}\label{lm_s0-s1-zeta-t0-unit}
For any compactefied curve $(\bC,\ovbC,\bCinf,\bO(1))$ with sections $S_0,S_1\colon U\to\bC$, $S_0(u)=S_1(u)=x$
there is a subcurve 
$(\sbC,\ovbC,\sbCinf,\bO(1)^\prime)$,
such that 
   images of $S_0$ and $S_1$ lies in $\sbC$,
  there are
  linear bundles $\bL_0$, $\bL_1$, $\bL$ on $\ovbC$ 
  and sections and isomorphisms $s_0$,$s_1$,$\zeta$,$\tau$ satisfying conditions \eqref{eq-sectcond1} and 
  in addition there are following sections 
\begin{equation*}\label{eq-sectcond2}\begin{aligned}
unit\in\Gamma(\ovbC,\bO(1)^\prime), Z(unit)\subset \sbCinf, \\
triv\in\Gamma(\ovbC,\bL_0(1)^\prime) , \, Z(triv)\subset \sbCinf, \\
\end{aligned}\end{equation*} 
(where $\bL_0(1)^\prime=\bL_0\otimes \bO(1)^\prime$)
\end{lemma}

\begin{lemma}\label{lm_s0s1zeta->es-ds}
Let $(\bC,\ovbC,\bCinf,\bO(1))$ be compactified curve with sections $S_0,S_1\colon U\to\bC$, $S_0(u)=S_1(u)=x$, 
and $\bL_0,\bL_1,\bL$, $s_0,s_1,\zeta,\tau$ be line bundles , sections , and isomorphisms satisfying \eqref{eq-sectcond1}.
Then for some large integer $l$  there are sections
\begin{equation}\label{eq-esds}\begin{aligned}
d\in\Gamma(\ovbC,\bO(l)), Z(d)\subset \sbC-x, \\
 e\in \Gamma(\ovbC,\bL(l)), Z(e)\subset \sbC-x, \\
(e\cdot d^{-1})\big|_{\sbCinf}^2= s_0\cdot s_1^{-1}, e\cdot d^{-1}\big|_{x}= \zeta\big|_{x}
\end{aligned}\end{equation}
(where the first equality is considered up to isomorphism $\tau\otimes \bL_1^{-1}\colon \bL\simeq \bL_0\otimes \bL_1^{-1}$)
And these equalities imply  
 that $$(\tau(s_0)\cdot d^{2})\big|_{\sbCinf\coprod 2x}= (s_1\cdot e^{2})\big|_{\sbCinf\coprod 2x}$$
(where $2x$ denotes closed subscheme $Spec\,k[C]/I(x)^2$ of closed fibre $C$ of $\bC$)

\end{lemma}

\begin{proof}[Proof of lemma \ref{lm_s0s1zeta}]
Let's denote $\bL_0=\bL(S_0)$ and $\bL_1=\bL(S_1)$ and choose any sections $s_0\in \Gamma(\ovbC,\bL_0)$ $Z(s_0)  = S_0$ and $s_0^\prime\in Gamma(\ovbC,\bL_1)$ $Z(s_1) = S_1$.
Since $S_0$ and $S_1$ are coincide at closed fibre with the point $x\in \ovC$ 
then $\bL_0\big|_{\ovC}\simeq\bL_1\big|_{\ovC}\simeq \bL(x)$.
By rigidity for Picard group there is a line bundle $\bL$ trivial on $\ovC$ with an 
isomorphism $\tau\colon \bL_0\simeq\bL_1\otimes\bL^2$. 

Let's then choose any trivialisation (invertible section) $\zeta\in \Gamma(\ovC,\bL\big|_{\ovC})$ and consider fraction $$\overline{\lambda}=(\tau(s_0)\cdot {s_1^\prime}^{-1})\big|_{\ovC}\cdot \zeta^{-2}.$$  
Formally by definition it is rational section of $\bO(\ovbC)\,(\simeq \bL_0\otimes\bL_1^{-1}\otimes\bL^{-2}\big|_{\ovC})$ (i.e. rational function) on projective curve $\ovC$ over $k(u)$. 
But since $div\,s_0\big|_{\ovC}=div\,s_1\big|_{\ovC}=x$ and hence $div\,s_0/s_1^\prime =0$ then $s_0/s_1^\prime\big|_{\ovC}$ is regular function. More precisely it is global section of $\bL_0\otimes\bL_1^{-1}$ and $\overline{\lambda}$ is regular function. And since any regular function on projective curve is constant then $\overline{\lambda}\in k(u)$. Let's lift  it to an invertible function $\lambda\in k(U)$ and define $s_1$ be $s_1^\prime\cdot \lambda$.
Then $(\tau(s_0)\cdot s_1^{-1})\big|_{\ovC}\cdot \zeta^{-2}$ becomes equal to $\overline{\lambda}\cdot \lambda^{-1}\big|_{\ovC} = 1$ that means that $\tau(s_0) = s_1\cdot \zeta^2$  

\end{proof}
\begin{proof}[Proof of lemma \ref{lm_s0-s1-zeta-t0-unit}]
Firstly by lemma \ref{lm_s0s1zeta} we find line bundles
$\bL_0^\prime,\bL_1^\prime,\bL$
sections 
$s_0^\prime, s_1^\prime , \zeta$ and isomorphism $\tau$.

Next let $F$ be a finite set of closed points in closed fibre $\ovC$ for one in each irreducible component.
The Serre theorem on cohomology of powers of ample bundle
yield that for all integer $l$ large some $L$ restriction homomorphisms 
$$\Gamma(\ovbC,\bO(l))\to\Gamma(F\cup x,\bO(l))\simeq \bigoplus\limits_{y\in F\cup x} k(y), \;
\Gamma(\ovbC,\bL_0(l))\to \Gamma(F\cup x,\bL_0(l))\simeq \bigoplus\limits_{y\in F\cup x} k(y)$$
are surjective and hence there are sections 
$$\begin{aligned}
  unit\in \Gamma(\ovbC,\bO(l)), 
  triv\in \Gamma(\ovbC, \bL_0(l)\colon \\
  unit(x),triv(y)\neq 0, y\in F\cup x
\end{aligned}$$ 
Then zero locus of $unit$ and $triv$ does not intersects with $S_0 \cup S_1$ and they are finite over $U$ since they closed fibres has dimension zero. So we can put 
$$
\sbCinf:= Z(unit)\cup Z(triv)\cup\bCinf,\, \bO(1)^\prime = \bO(l),\, \sbC = \ovbC-\sbCinf
$$
and get
$$\begin{aligned}
unit\in \Gamma(\ovbC,\bO(1)^\prime),\, triv\in \Gamma(\ovbC, \bL_0(1)^\prime),\,
Z(unit),Z(t_0)\subset \bC^\prime
\end{aligned}$$

\end{proof}

\begin{proof}[Proof of lemma \ref{lm_s0s1zeta->es-ds}]
The claim immediately follows from three points, namely 
from that $\bCinf$ is henselian,
section $s_0\cdot s_1^{-1}\big|_{\Cinf}$ is square 
and by the Serre theorem   restriction homomorphisms 
\begin{equation}\label{eq_s-rest-hom}
\Gamma(\ovbC,\bO(l))\to \Gamma(S_0\cup \bCinf,\bO(l))
\Gamma(\ovbC,\bL(l))\to \Gamma(S_1\cup \bCinf,\bL(l))
\end{equation}
are surjective
for some large $l$.

In fact by conditions on sections $s_0,s_1,\zeta$ \eqref{eq_sectcond1} 
$s_0\cdot s_1^{-1}\big|_{\ovC} = \zeta^2$ and in particular $s_0\cdot s_1^{-1}\big|_{\Cinf} = \zeta^2\big|_{\Cinf}$. 
Since $\bCinf$ is finite over henselian scheme $U$ it is henselian.
So there is a section $$\overline{s}\in\Gamma(\bCinf,\bL)\colon \overline{s}^2 = s_0\cdot s_1^{-1} .$$
 
Now let's choose any decomposition of $\overline{s}$ into fractions of invertible sections
$$\overline{s} = {\overline{e}}\cdot{\overline{d}^{-1}}, \overline{e}\in\Gamma(\bCinf,\bO(l)), \overline{d}\in\Gamma(\bCinf,\bL(l)).$$
Then by surjectivity of homomorphisms \eqref{eq_s-rest-hom} there are sections
$$\begin{aligned}
e\in \Gamma(\ovbC,\bO(l)), e\big|_{\bCinf}=\overline{e}, e\big|_{S_0} = unit^l\big|_{S_0}\\
d\in \Gamma(\ovbC,\bL(l)), d\big|_{\bCinf} = \overline{d}, d\big|_{S_1} = \zeta\cdot unit^{l}\big|_{S_1} 
\end{aligned}$$
that implies conditions \eqref{eq-esds}.

The equality
$(\tau(s_0)\cdot d^{2})\big|_{\sbCinf}= (s_1\cdot e^{2})\big|_{\sbCinf}$
is the same to first equality of \eqref{eq-esds}
and to conclude 
$(\tau(s_0)\cdot d^{2})\big|_{2x}= (s_1\cdot e^{2})\big|_{2x}$
from the second it is enough to note that by \eqref{eq-sectcond1}
$\tau(s_0)\big|_{x}=s_1\cdot \zeta^2\big|_{x}$
 
\end{proof}

{\em Now we proceed to prove the theorem.}\label{proof_theoremRWCinner}
\par 
Let $\bO(1)$ be an ample bundle on relative projective curve $\ovbC$, 
then data $(\bC,\ovbC,\bCinf,\bO(1))$ from the formulation of theorem is compactified curve over $U$.

Firstly, let's
choose subcurve $(\sbC,\ovbC,\sbCinf,\bO(1))$ of $(\bC,\ovbC,\bCinf,\bO(1))$ (see definition \ref{def_subcomcurve} 
such that $\omega(\bC)$ is trivial. 
And choose some trivialisation $\mu \colon \omega(\sbC)\otimes\omega(U)^{-1}\simeq \bO(\sbC)$
(To do this it is enough to choose $\sbC$ as enough small open Zarisky neighbourhood in $\bC$ of the finite set of points in closed fibre, including the point $x$ and for one point in each irreducible component)

Next consequently applying lemmas 
  \ref{lm_s0-s1-zeta-t0-unit} and \ref{lm_s0s1zeta->es-ds} to $(\sbC,\ovbC,\sbCinf)$
we get subcurve $(\ssbC,\ovbC,\ssbCinf,\bO(1)^\prime)$ of the compactified curve  $(\sbC,\ovbC,\sbCinf,\bO(1))$ 
  with linear bundles $\bL_0$, $\bL_1$, $\bL$, integer $l$ and sections and isomorphisms
\begin{gather*}
\begin{aligned}
s_0\in \Gamma(\ovbC,\bL_0) &\colon Z(s_0)=S_0, &
 s_1\in \Gamma(\ovbC,\bL_1) &\colon  Z(s_1)=S_1,\\
\zeta\in \Gamma(\ovC,\bL\big|_{\ovC}), & \;\tau\colon \bL_0\simeq\bL_1\otimes\bL^2,
  &  \tau(s_0)\big|_{\ovC} = s_1\big|_{\ovC}\cdot \zeta^2, \\
triv\in \Gamma(\ovbC,\bL_0(1))&\colon Z(triv)\subset\ssbCinf, &
unit\in \Gamma(\ovbC,\bO(1)^\prime) &\colon Z(unit)\subset\ssbCinf,\\
e\in \Gamma(\ovbC,\bL_0(l)^\prime)&\colon Z(e)\subset \ssbC-x, &
d\in \Gamma(\ovbC,\bO(l)^\prime)&\colon Z(d)\subset \ssbC-x,
\end{aligned}\\
(\tau(s_0)\cdot d^2)\big|_{\ssbCinf\amalg 2x} = (s_1\cdot e^2)\big|_{\ssbCinf\coprod 2x}
\end{gather*} 
(where under $2x$ in the last row we mean nilpotent closed subscheme $Spec(k[\ssbC]/I^2(x))\subset \ssbC$).
And since $\ssbC\subset\sbC$ its canonical class is trivial, so let's choose some trivialisation $\mu\colon \omega(\ssbC)\simeq\bO(\ssbC)$. 

The statement of theorem for any subcurve  
implies the statement for original $(\bC,\ovbC,\bCinf,\bO(1))$, 
i.e. $[S_0]-[S_1]=0\in \overline{WCor_k}(U,\ssbC)\implies [S_0]-[S_1]=0\in \overline{WCor_k}(U,\bC)$, 
so let's redenote $(\bC,\ovbC,\bCinf,\bO(1))$ to $(\ssbC,\ovbC,\ssbCinf,\bO(1)^\prime)$.

Now let's consider regular map
\begin{equation*}
\overline{h} = 
  [\tau(s_0)\cdot d^2\cdot (1-\lambda) + s_1\cdot e^2 \cdot\lambda \colon \tau(triv)\cdot unit^{2l-1}]\colon 
  \ovbC\to \prl
\end{equation*}  
that is well defined because
 $(\tau(s_0)\cdot d^2)\big|_{\sbCinf}=(s_1\cdot e^2)\big|_{\bCinf}$ is invertible
and 
 $Z(triv\cdot unit^{2l-1})=Z(\tau(triv)\cdot unit^{2l-1})\subset \bCinf$
 (so vanish locus of denominator and enumerator don't intersect)
Applying to this map and trivialisation $\mu$ 
  the construction \ref{const_OrWLf-WCor}
we get quadratic space $\Pi = (k[Z],q)$ in $\mathcal P(Z_{U\times\affl})$ where $Z = h^{-1}(0)\subset \bC\times\affl$
defining morphism $\Phi \langle U,\ovbC,\bC,\mu,h,id_{\bC}\rangle\in WCor_k(U\times\affl, \bC)$.

We are going to show that after some twist by invertible function on $U$ it defines homotopy between $[S_0]$ ands $[S_1]$ in $WCor_k(U,\bC)$. 
Let $Pi_0=i_0^*(\Pi)$, and $\Pi_1=i_1^*(\Pi)$, where $i_0,i_1\colon U\to U\times \affl$ be zero and unit sections.

To do this let's note that
the fibre of $Z$ over $U\times 0$ is $Z(s_0\cdot d^2)$ and since 
$Z(s_0)=S_0$ and $S_0 \cap Z(d) = \emptyset$
it splits into $S_0 \coprod Z(d^2)$.
Hence quadratic space $\Pi_0$ splits into the direct sum $(k[S_0],q_0)\oplus (k[Z(d^2)],q^\prime_0)$.
By lemma \ref{lm_SqMet} the secons summand is metabolic and defines zero morphism in $WCor(U,\bC)$,
and the first summand is quadratic space of the rank one (over $k[U]$), so we get
$$\Phi_0\stackrel{def}{=}\Phi\circ i_0 = [S_0]\circ \langle q_0 \rangle, q_0\in k[U]^*$$.
(see definition \ref{def_twaut} for $\langle q_0\rangle$) 
And similarly  
$$\Phi_1\stackrel{def}{=}\Phi\circ i_1 = [S_1] \circ \langle q_1 \rangle, q_0\in k[U]^*$$.
So $h$ together with $\mu$ give us homotopy in $WCor$ connecting $[S_0]\circ \langle q_0\rangle$ and $[S_0]\circ \langle q_0\rangle$. 

Let's consider closed fibre of this homotopy, i.e. $\Pi_u=i_u^*(\Pi)= (k[Z\times_U u],q_u)$ where $i_u\colon u\times\affl\hookrightarrow U\times\affl$. 
Since $(\tau(s_0)\cdot d^2)\big|_{x} = (s_1\cdot e^2)\big|_{x} = 0$ and 
  $(\tau(s_0)\cdot d^2)\big|_{2x} = (s_1\cdot e^2)\big|_{2x} \neq 0$ , 
$\affl\times x$ is isolated component of $Z\times_U u$.
Hence $\Pi_x$ splits into sum $(k[x\times\affl],q_x)\oplus (k[Z\times_U u - x\times\affl], q_u^\prime)$,
and $q_x$ is invertible function on $\affl\times x$, hence it is constant in $k(u)*$.

Thus we get that fibre of $\Pi$ over $u\times 0$ splits into the sum of quadratic spaces with support $x\times 0$ and some other one (don't contain $x\times 0$) and on the one side 
the first summand is equal to 
$(k[x\times 0], q_0(u) )$ 
and on the other side -- to
$(k[x,\times 0],q_x)$. 
Hence $q_0(u) = q_x\cdot r_0^2$ where $r_0\in k(u)^*$.
And similarly $q_1(u) = q_x\cdot r_1^2$ where $r_1\in k(u)^*$.

So 
$q_1\cdot q_0^{-1}$ is square using lemma-remark \ref{lm_unitclp-symb=unit} we get that after multiplication of quadratic form $q$ of $\Pi$ by $q_0^{-1}$, i.e. composing of homotopy $\Phi$ with automotphism $\langle q_0\rangle \in WCor(U,U)$ we get homotopy between $[S_0]$ and $[S_1]$.

\end{proof}

\section{Appendix}
\begin{theorem}\label{cth_PicRigi}
For any relative smooth projective curve $\ovbC$ over henselian local scheme $U$ that is henselisation at closed point of a smooth variety over field $k$,
$$Pic(\ovbC)/n\hookrightarrow Pic(\ovC)/n,\forall n\in \mathbb Z, (n,char\,k)=1$$
where $\ovC$ denotes closed fibre of $\ovbC$.

\end{theorem}
And more precisely it can be reformulated as the following 
\begin{theorem}
In the situation of theorem \ref{cth_PicRigi} for any linear bundle $\bL$ on $\ovbC$ 
such that $\bL\big|_{\ovC}\simeq \bO(\ovC)$  there is line bundle bundle $\bL^\prime $ such that $\bL\simeq {\bL^\prime}^{\otimes n}$.
\end{theorem}

\begin{proof}
Since $Pic(X)\simeq H^1(X,\mu)$ for any variety $X$
exact sequence of groups $\mu_n\hookrightarrow \mu\xrightarrow{x\mapsto x^n}\mu$
leads to following commutative diagram with rows being exact sequences
\begin{gather*}
\xymatrix{
Pic(\ovbC)\ar[r]^{\dot n}\ar[d] &Pic(\ovbC)\ar[r]\ar[d] &H^2(\ovbC,\mu_n)\ar[d]\\
Pic(\ovC)\ar[r]^{\dot n}             &Pic(\ovC)\ar[r]            &H^2(\ovC,\mu_n)
}
\end{gather*}
Right vertical arrow $H^2(\ovbC,\mu_n)\to H^2(\ovC,\mu_n)$ is injective by proper base change theorem.
So we get diagram 
$$\xymatrix{
Pic(\ovbC)/n Pic(\ovbC) \ar@{^(->}[r]\ar[d] &H^2(\ovbC,\mu_n)\ar@{^(->}[d]\\
Pic(\ovC) / n  Pic(\ovC)\ar@{^(->}[r]            &H^2(\ovC,\mu_n)
}
$$
And since horizontal arrows 
and right vertical arrow  are injective,
hence left vertical arrow is injective too.

\end{proof}

\label{bib}

\Address

\end{document}